\documentclass{article}

\usepackage{amsmath,amsthm,bm,colortbl,graphicx,mathrsfs,amssymb,afterpage,url,cite,verbatim,ulem, multirow, subfig, rotating, booktabs, caption}

\theoremstyle{plain}
\newtheorem{theorem}{Theorem}
\newtheorem*{theorem1'}{Theorem 1$'$}
\newtheorem*{theorem2'}{Theorem 2$'$}

\def\P{{\rm P}} 
\def\E{{\rm E}} 
\def\eps{\varepsilon}
\allowdisplaybreaks

\makeatother
\title{The tilted flashing Brownian ratchet}
\author{S. N. Ethier\thanks{Department of Mathematics, University of Utah, 155 S. 1400 E., Salt Lake City, UT 84112, USA. e-mail: ethier@math.utah.edu.}\; and Jiyeon Lee\thanks{Department of Statistics, Yeungnam University, 280 Daehak-Ro, Gyeongsan, Gyeongbuk 38541, South Korea.  e-mail: leejy@yu.ac.kr.}}
\date{}

\begin{document}

\maketitle

\begin{abstract}
The flashing Brownian ratchet is a stochastic process that alternates between two regimes, a one-dimensional Brownian motion and a Brownian ratchet, the latter being a one-dimensional diffusion process that drifts towards a minimum of a periodic asymmetric sawtooth potential.  The result is directed motion.  In the presence of a static homogeneous force that acts in the direction opposite that of the directed motion, there is a reduction (or even a reversal) of the directed motion effect.  Such a process may be called a tilted flashing Brownian ratchet.  We show how one can study this process numerically, using a random walk approximation or, equivalently, using numerical solution of the Fokker--Planck equation.  Stochastic simulation is another viable method.
\end{abstract}

\section{Introduction}\label{intro}

The flashing Brownian ratchet, introduced by Ajdari and Prost~\cite{AP92}, is a stochastic process that alternates between two regimes, a one-dimensional Brownian motion and a Brownian ratchet, the latter being a one-dimensional diffusion process that drifts towards a minimum of a periodic asymmetric sawtooth potential.  The result is directed motion, as shown in panels ($a$)--($c$) of Figure~\ref{HAT00fig} (from Harmer et al.~\cite{HAT00}).  This conceptual figure, specifically panels ($a$)--($c$), although largely accurate, can be improved, as demonstrated by Ethier and Lee~\cite{EL18} using a random walk approximation.

\begin{figure}[th]
\centering
\includegraphics[width=4.5in]{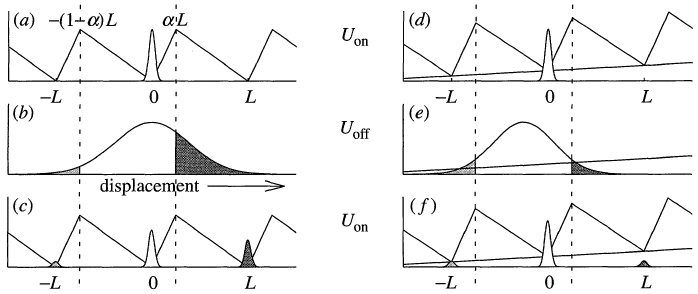}
\caption{This shows how the mechanism of the ratchet potential works. The diagrams on the left, ($a$)--($c$), show when there is no macroscopic gradient present and the net movement of particles is in the forward direction (defined by arrow). The diagrams on the right, ($d$)--($f$), have a slight gradient present; this causes the particles to drift backwards while $U_{\text{off}}$ is acting. Hence the net flow of particles in the forward direction is reduced. (Figure and caption reprinted from Harmer et al.~\cite{HAT00} with the permission of the Royal Society.  A nearly identical figure appeared in Harmer and Abbott~\cite{HA99}.)}
\label{HAT00fig}
\end{figure}

In the presence of a static homogeneous force that acts in the direction opposite that of the directed motion, there is a reduction (or even a reversal) of the directed motion effect, as illustrated in panels ($d$)--($f$) of Figure~\ref{HAT00fig}.  (For alternative figures, see Astumian~\cite[Fig.~2A]{A97} and Astumian and H\"anggi~\cite[Fig.~2]{AH02}.)  We refer to such a process as a \textit{tilted flashing Brownian ratchet}, and it is our aim here to study this process numerically, much as we did for the flashing Brownian ratchet~\cite{EL18}, using a random walk approximation.  We will see, in particular, that panels ($d$)--($f$) of the conceptual Figure~\ref{HAT00fig}, although again largely accurate, can be improved as well.

We formulate the model using the notation of Figure~\ref{HAT00fig}.  First, the asymmetric sawtooth potential $V$ can be defined by
\begin{equation}\label{sawtooth}
V(x):=\begin{cases} x/\alpha&\text{if $0\le x\le\alpha L$,}\\(L-x)/(1-\alpha)&\text{if $\alpha L\le x\le L$,}\end{cases}
\end{equation}
extended periodically, with period $L$, to all of ${\bf R}$.  Here $0<\alpha<1$ and $L>0$, and $\alpha\ne1/2$ by virtue of the asymmetry.  ($\alpha$ is a shape parameter, and $L$ is the period and the amplitude of the potential.) The \textit{tilted Brownian ratchet} is a one-dimensional diffusion process with diffusion coefficient 1 and drift coefficient $\mu$ of the form, for some $\gamma>0$ and $\kappa$ real,
\begin{equation}\label{mu(x)}
\mu(x):=-[\gamma V'(x)+\kappa]=\begin{cases}-\gamma/\alpha-\kappa&\text{if $0\le x<\alpha L$,}\\ \gamma/(1-\alpha)-\kappa&\text{if $\alpha L\le x< L$,}\end{cases}
\end{equation}
again extended periodically, with period $L$, to all of ${\bf R}$.  
Such a process $X_t$ is governed by the It\^o stochastic differential equation (SDE)
\begin{equation}\label{SDE1}
dX_t=-[\gamma V'(X_t)+\kappa]\,dt+dB_t,
\end{equation}
where $B_t$ is standard Brownian motion.  This diffusion process drifts to the left on $(nL,nL+\alpha L)$ and drifts to the right on $(nL-(1-\alpha)L,nL)$, for each $n\in{\bf Z}$, provided $-\gamma/\alpha<\kappa<\gamma/(1-\alpha)$.  In other words, it drifts towards a minimum of the sawtooth potential $\gamma V$ (or a local minimum of the tilted sawtooth potential $\gamma V(x)+\kappa x$).  $\gamma L$ is the amplitude of the potential and $\kappa$ is the slope of the tilt. 

While the inequalities $-\gamma/\alpha<\kappa<\gamma/(1-\alpha)$ are not needed to define the process, they are needed for the tilted ratchet to remain a ratchet.  Furthermore, we usually assume that the static homogeneous force acts in the direction opposite that of the directed motion, that is, $\big(\frac12-\alpha\big)\kappa>0$, but this is not essential either.  

It may be helpful to reconcile our model with those in the literature.  Our tilted Brownian ratchet is equivalent to what Reimann~\cite[Eq.~(2.34)]{R02} called a ``tilted Smoluchowski--Feynman ratchet,'' acknowledging two of the pioneers of the subject.  It is given by the Langevin equation
\begin{equation}\label{Langevin}
\eta\, \dot{x}(t)=-\beta V'(x(t))+F+\sqrt{2\eta k_\text{B}T}\,\xi(t),
\end{equation}
where $x(t)$ is the position of the particle at time $t$, $V$ is the periodic sawtooth potential~\eqref{sawtooth}, $\beta L$ is the amplitude of the potential ($\beta$ does not appear in \cite{R02} because its $V$ is our $\beta V$),  $F$ is an additional static homogeneous force, $\xi(t)$ is a Gaussian white noise with mean $\langle\xi(t)\rangle=0$ and covariance $\langle \xi(t)\xi(s)\rangle=\delta(t-s)$, $\eta$ is a friction coefficient, $k_\text{B}$ is Boltzmann's constant, and $T$ is temperature.  We interpret this equation mathematically as the SDE
$$
\eta\, dx(t)=[-\beta V'(x(t))+F]\,dt+\sqrt{2\eta k_\text{B}T}\,db(t),
$$
where $b(t)$ is standard Brownian motion, or
\begin{equation}\label{SDE-b}
dx(t)=\eta^{-1}[-\beta V'(x(t))+F]\,dt+\sigma\,db(t)
\end{equation}
with $\sigma=\sqrt{2\eta^{-1} k_\text{B}T}$.  With the time change $X_t:=x(\sigma^{-2}t)$ and standard Brownian motion $B_t:=\sigma\, b(\sigma^{-2}t)$, \eqref{SDE-b} becomes
$$
dX_t=\sigma^{-2}\eta^{-1}[-\beta V'(X_t)+F]\,dt+dB_t,
$$
which is \eqref{SDE1} with 
$$
\gamma=\sigma^{-2}\eta^{-1}\beta=\frac{\beta}{2k_\text{B}T}\quad\text{and}\quad \kappa=-\sigma^{-2}\eta^{-1}F=-\frac{F}{2k_\text{B}T}.
$$  
Other sources with equations similar to \eqref{Langevin} include Reimann and H\"anggi~\cite[Eq.~(11)]{RH02}, Parrondo and de Cisneros~\cite[Eq.~(22)]{PD02}, Dinis~\cite[Eq.~(1.88) with $\alpha(t)=1$]{D06}, and H\"anggi and Marchesoni~\cite[Eq.~(3)]{HM09}. 

Given time parameters $\tau_1,\tau_2>0$, the \textit{tilted flashing Brownian ratchet} is a time-inhomogeneous one-dimensional diffusion process that evolves as a Brownian motion with drift $-\kappa$ (i.e., $B_t-\kappa t$) on $[0,\tau_1]$ (potential ``off''), then as the tilted Brownian ratchet~\eqref{SDE1} on $[\tau_1,\tau_1+\tau_2]$ (potential ``on''), and so on, alternating between these two regimes.  Such a process $Y_t$ is governed by the SDE
\begin{equation}\label{SDE2}
dY_t=-[\gamma\zeta(t) V'(Y_t)+\kappa]\,dt+dB_t,
\end{equation}
where\footnote{The quantity $\text{mod}(t,\tau)$ is defined as the remainder (in $[0,\tau)$) when $t$ is divided by $\tau$.}
\begin{equation}\label{zeta(t)}
\zeta(t):=\begin{cases}0&\text{if mod$(t,\tau_1+\tau_2)<\tau_1$,}\\
                       1&\text{if mod$(t,\tau_1+\tau_2)\ge\tau_1$.}\end{cases}
\end{equation}

The flashing Brownian ratchet is the process that motivated Parrondo's paradox (Harmer and Abbott~\cite{HA99,HA02}), in which two fair games of chance, when alternated in some way, produce a winning (or losing) game. The tilted flashing Brownian ratchet, which is a flashing Brownian ratchet in the presence of a static homogeneous force, can be discretized to yield a stronger form of Parrondo's paradox, in which two losing games combine to win (or two winning games combine to lose).  In Section~\ref{Parrondo-games} we provide a general formulation of Parrondo's paradox motivated by the tilted flashing Brownian ratchet.  These capital-dependent Parrondo games are modified in Section~\ref{approx} so as to yield our random walk approximation, which we then suggest can be improved.  Other ways of numerically studying the tilted flashing Brownian ratchet include stochastic simulation using the Euler--Maruyama method (Section~\ref{EM}), and numerical solution of the Fokker--Planck equation using the finite-difference method (Section~\ref{FP}). We show in Sections~\ref{modeling1} and \ref{modeling2} that the conceptual Figure~\ref{HAT00fig} is not an entirely accurate representation of the behavior of the tilted flashing Brownian ratchet.  This conclusion can be deduced from the random walk approximation, from stochastic simulation, or from numerical solution of the Fokker--Planck equation.

\section{Parrondo games and tilted Brownian ratchets}
\label{Parrondo-games}

Parrondo's paradox is a discretization and reinterpretation of the flashing Brownian ratchet, as explained in \cite{EL18} and in various other references \cite{HAT00, HA99, HA02, AA02, HATP01, TAM03a, TAM03b, HKK02}. The role of displacement is played by profit in a game of chance.  The role of Brownian motion is played by game $A$, whose cumulative-profit process is a simple symmetric random walk on $\textbf{Z}$.  The role of the Brownian ratchet is played by game $B$, whose cumulative-profit process is a simple random walk on $\textbf{Z}$ with periodic state-dependent transition probabilities.  The role of the flashing Brownian ratchet is played by some combination of games $A$ and $B$, either a random mixture (i.e., $cA+(1-c)B$, where $0<c<1$) or a nonrandom periodic pattern (e.g., $AABB\,AABB\,\cdots$).  Finally, the role of tilt is played by a bias parameter in the transition probabilities of the random walks.

Let $0<\alpha<1$ and assume that $\alpha$ is rational, so that $\alpha=l/L$ for some relatively prime positive integers $1\le l<L$.  The cumulative-profit process of game $B$ is the simple random walk on $\textbf{Z}$ with periodic state-dependent transition probabilities of the form  
\begin{equation}\label{transitions2}
P(j,j+1):=\begin{cases}p_0&\text{if mod$(j,L)<l$,}\\ p_1&\text{if mod$(j,L)\ge l$,}\end{cases}
\end{equation}
and $P(j,j-1)=1-P(j,j+1)$, where $0<p_0<1/2<p_1<1$.  Because of the period-$L$ transition probabilities, the unique reversible invariant measure $\pi$ must have period $L$ (i.e., $\pi(j)=\pi(j+L)$) for the random walk to be recurrent.  We can confirm that the detailed balance conditions have a solution if and only if $(1-p_0)^l(1-p_1)^{L-l}=p_0^l p_1^{L-l}$.  Solving for $p_1$, we obtain
\begin{equation*}
\label{p1}
p_1=\frac{1}{1+[p_0/(1-p_0)]^{\alpha/(1-\alpha)}}.  
\end{equation*}
With the $\alpha/(1-\alpha)$th power in the denominator denoted by $\rho$, the requirements that $0<p_0<1/2<p_1<1$ become $0<\rho<1$, and
\begin{equation}\label{p0,p1-2}
p_0=\frac{\rho^{(1-\alpha)/\alpha}}{1+\rho^{(1-\alpha)/\alpha}}\quad\text{and}\quad p_1=\frac{1}{1+\rho}.
\end{equation}
As shown in \cite{EL18}, game $B$ is fair (asymptotically), whereas the random mixture $cA+(1-c)B$, where $0<c<1$, is winning if $\alpha<1/2$ and losing if $\alpha>1/2$.  Thus, two fair games combine to win if $\alpha<1/2$, which is the \textit{Parrondo effect}, and  combine to lose if $\alpha>1/2$, which is the \textit{anti-Parrondo effect}.  Both are examples of Parrondo's paradox.  Analysis of nonrandom periodic patterns is more complicated.

A slightly stronger form of the Parrondo effect occurs when two losing games combine to win.  That can be achieved by replacing the cumulative-profit process of game $A$ by the simple random walk on $\textbf{Z}$ with 
\begin{equation}\label{p}
P(j,j+1)=p:=\frac12-\varepsilon
\end{equation}
and $P(j,j-1)=1-P(j,j+1)$, where the bias parameter $\varepsilon>0$ is sufficiently small,
and replacing the cumulative-profit process of game $B$ by the simple random walk on $\textbf{Z}$ with periodic state-dependent transition probabilities of the form \eqref{transitions2} with $\alpha<1/2$ but with \eqref{p0,p1-2} replaced by
\begin{equation}\label{p0,p1}
p_0=\frac{\rho^{(1-\alpha)/\alpha}}{1+\rho^{(1-\alpha)/\alpha}}-\varepsilon\quad\text{and}\quad p_1=\frac{1}{1+\rho}-\varepsilon.
\end{equation}
Parrondo's original formulation of the games assumed $\alpha=1/3$, $L=3$, $\rho=1/3$, and $\varepsilon=1/200$.

A slightly stronger form of the anti-Parrondo effect occurs when two winning games combine to lose.  That can be achieved as in the preceding paragraph but with $\alpha>1/2$ and $\varepsilon<0$, where $|\varepsilon|$ is sufficiently small.

\section{Random walk approximation}\label{approx}

We have seen that a discretization of the tilted flashing Brownian ratchet yields capital-dependent Parrondo games with a bias parameter.  In this section we demonstrate the converse:  A continuization of capital-dependent Parrondo games with a bias parameter yields the tilted flashing Brownian ratchet.

Let $0<\alpha<1$ and assume that $\alpha$ is rational, so that $\alpha=l/L$ for some relatively prime positive integers $1\le l<L$.  Consider a sequence of simple random walks on ${\bf Z}$ with periodic state-dependent transition probabilities defined as follows.  For each $n\ge1$, we let
\begin{equation}\label{transitions3}
P_n(j,j+1):=\begin{cases}p_0&\text{if mod$(j,nL)<nl$,}\\ p_1&\text{if mod$(j,nL)\ge nl$,}\end{cases}
\end{equation}
and $P_n(j,j-1)=1-P_n(j,j+1)$, where $p_0$ and $p_1$ are as in \eqref{p0,p1}.  
The special case of \eqref{transitions3} in which $n=1$ is precisely \eqref{transitions2}.

We will let $n\to\infty$ but first we require 
\begin{equation}\label{rho,eps}
\rho=1-\frac{\lambda}{n} \quad\text{and}\quad \eps=\frac{\kappa}{2n},
\end{equation}
where $\lambda>0$ and $\kappa$ is real, then we rescale time by allowing $n^2$ jumps per unit of time, and finally we rescale space by dividing by $n$.  The result in the limit as $n\to\infty$ is a tilted Brownian ratchet.  

We denote by $D_{\bf R}[0,\infty)$ the space of real-valued functions on $[0,\infty)$ that are right-continuous with left limits, and we give it the Skorokhod topology.

\begin{theorem}[{\bf Random walk approximation of tilted Brownian ratchet}\label{theorem1}] 
\textit{Let $\lambda>0$ and $\kappa$ be real.  For $n=1,2,\ldots$ $($and $n$ sufficiently large that $\rho,p_0,p_1\in(0,1))$, let $\{X_n(k),\,k=0,1,\ldots\}$ denote the random walk on ${\bf Z}$ defined by \eqref{p0,p1}--\eqref{rho,eps}, and let $\{X_t,\, t\ge0\}$ denote the tilted Brownian ratchet with parameters $\gamma:=\lambda(1-\alpha)/2$ and $\kappa$.  If $X_n(0)/n$ converges in distribution to $X_0$ as $n\to\infty$, then $\{X_n(\lfloor n^2 t\rfloor)/n,\,t\ge0\}$ converges in distribution in $D_{\bf R}[0,\infty)$ to $\{X_t,\, t\ge0\}$ as $n\to\infty$.}
\end{theorem}

\begin{proof}
The argument is essentially as in the proof of Theorem~1 of \cite{EL18}, except that, if $\mu_n\to\mu$ as $n\to\infty$, then
\begin{align*}
\frac12\bigg(1+\frac{\mu_n}{n}\bigg)=p_0&=\frac{(1-\lambda/n)^{(1-\alpha)/\alpha}}{1+(1-\lambda/n)^{(1-\alpha)/\alpha}}-\frac{\kappa}{2n}\\
&=\frac12\bigg(1-\frac{\lambda(1-\alpha)}{2n\alpha}-\frac{\kappa}{n}+o(n^{-1})\bigg)
\end{align*}
leads to $\mu=-\lambda(1-\alpha)/(2\alpha)-\kappa=-\gamma/\alpha-\kappa$, and
\begin{align*}
\frac12\bigg(1+\frac{\mu_n}{n}\bigg)=p_1&=\frac{1}{1+(1-\lambda/n)}-\frac{\kappa}{2n}\\
&=\frac12\bigg(1+\frac{\lambda}{2n}-\frac{\kappa}{n}+o(n^{-1})\bigg)
\end{align*}
leads to $\mu=\lambda/2-\kappa=\gamma/(1-\alpha)-\kappa$.
\end{proof}

Ordinarily, a Markov chain is given, and one attempts to approximate it by a diffusion process that is more amenable to analysis.  There are many examples in such areas as population genetics and queueing theory.  Here we are doing the opposite.  A diffusion process is given, and we want to approximate it by a random walk that is more amenable to computation.  The sequence of random walks in Theorem~\ref{theorem1} has a clear relationship to the capital-dependent Parrondo games of Section~\ref{Parrondo-games}, but it is not written in stone.  Perhaps it can be improved, in the sense of converging faster, thereby yielding a better approximation.  Towards this end we find a clue in the proof of Theorem~\ref{theorem1}.  We replace \eqref{p0,p1} and \eqref{rho,eps} by
\begin{equation}\label{p0,p1-improved}
p_0=\frac12\bigg(1-\frac{\lambda(1-\alpha)}{2n\alpha}-\frac{\kappa}{n}\bigg)\quad\text{and}\quad
p_1=\frac12\bigg(1+\frac{\lambda}{2n}-\frac{\kappa}{n}\bigg).
\end{equation}

\begin{theorem1'} [{\bf Improved random walk approximation of tilted Brownian ratchet}\label{theorem1-imprved}]
{\it
Theorem~\ref{theorem1} remains true with ``\eqref{p0,p1}--\eqref{rho,eps}'' replaced by ``\eqref{transitions3} and \eqref{p0,p1-improved}''.
}
\end{theorem1'}

For the next theorem, we assume that the time parameters $\tau_1,\tau_2>0$ of the tilted flashing Brownian ratchet are rational.  We let $m$ be the smallest positive integer such that $m^2\tau_1$ and $m^2\tau_2$ are integers.  Recall \eqref{p} and augment it with \eqref{rho,eps}.

\begin{theorem}[{\bf Random walk approximation of tilted flashing Brownian ratchet}\label{theorem2}]
\textit{Let $\lambda>0$, $\kappa$ be real, and $\tau_1$, $\tau_2$, and $m$ be as above.
For $n=m,2m,3m,\ldots$ $($and $n$ sufficiently large that $p,\rho,p_0,p_1\in(0,1))$, let $\{Y_n(k),\,k=0,1,\ldots\}$ denote the time-inhomogeneous random walk on ${\bf Z}$ that evolves as the simple random walk with $P_n^0(j,j+1)=p=1-P_n^0(j,j-1)$ for $n^2\tau_1$ steps, then as the random walk of Theorem~\ref{theorem1} for $n^2\tau_2$ steps, and so on, alternating in this way.  Let $\{Y_t,\, t\ge0\}$ denote the tilted flashing Brownian ratchet with parameters $\gamma:=\lambda(1-\alpha)/2$, $\kappa$, $\tau_1$, and $\tau_2$.  If $Y_n(0)/n$ converges in distribution to $Y_0$ as $n\to\infty$, then $\{Y_n(\lfloor n^2 t\rfloor)/n,\,t\ge0\}$ converges in distribution in $D_{\bf R}[0,\infty)$ to $\{Y_t,\, t\ge0\}$ as $n\to\infty$.  $($Here $n\to\infty$ through multiples of $m$.$)$}
\end{theorem}

\begin{proof}
The argument is essentially as in the proof of Theorem~2 of \cite{EL18}.
\end{proof}

\begin{theorem2'}[{\bf Improved random walk approximation of tilted flashing Brownian ratchet}\label{theorem2-improved}]
\textit{Theorem~\ref{theorem2} remains true with ``Theorem~\ref{theorem1}'' replaced by ``Theorem~1$'$''.}
\end{theorem2'}

It seems clear, and computations suggest, that the random walk approximation of Theorem~2$'$ is more accurate than that of Theorem~\ref{theorem2}, thereby justifying the use of the adjective ``improved''.

\section{Simulation of the solution of the SDE}\label{EM}

Here we approximate the solution of the SDE~\eqref{SDE1} by the solution of the stochastic difference equation (with time step $h$)
\begin{align*}
&X((k+1)h)-X(kh)\\
&\qquad{}=-[\gamma V'(X(kh))+\kappa]h+B((k+1)h)-B(kh),\quad k=0,1,\ldots,\nonumber
\end{align*}
often called the \textit{Euler--Maruyama method}.  Here the Brownian increments $B((k+1)h)-B(kh)$ are independent normal random variables with mean 0 and variance $h$.  To justify the method, we take $h=1/n^2$ and $X_n(k):=X(kh)$, obtaining the sequence of stochastic difference equations
\begin{equation}\label{SDE-discrete}
X_n(k+1)-X_n(k)=-\frac{\gamma V'(X_n(k))+\kappa}{n^2}+\frac{Z_{k+1}}{n},\quad k=0,1,\ldots,
\end{equation}
where $Z_1,Z_2,\ldots$ are independent standard normal random variables.

\begin{theorem}[{\bf Simulation of the tilted Brownian ratchet}]
\textit{Let $\gamma>0$ and $\kappa$ be real.  For $n=1,2,\ldots$, define the continuous-state Markov chain $\{X_n(k),\, k=0,1,\ldots\}$ by \eqref{SDE-discrete} $($given the initial state $X_n(0)$, independent of $Z_1,Z_2,\ldots)$, and let $\{X_t,\, t\ge0\}$ denote the tilted Brownian ratchet with parameters $\gamma$ and $\kappa$.  If $X_n(0)$ converges in distribution to $X_0$ as $n\to\infty$, then $\{X_n(\lfloor n^2 t\rfloor),\,t\ge0\}$ converges in distribution in $D_{\bf R}[0,\infty)$ to $\{X_t,\, t\ge0\}$ as $n\to\infty$.}
\end{theorem}

\begin{proof}
The proof is a straightforward application of Corollary 4.8.17 of \cite{EK86}.
\end{proof}

Again, we assume that the time parameters $\tau_1,\tau_2>0$ of the tilted flashing Brownian ratchet are rational.  We let $m$ be the smallest positive integer such that $m^2\tau_1$ and $m^2\tau_2$ are integers.  

We now modify \eqref{SDE-discrete} to
\begin{equation}\label{SDE2-discrete}
Y_n(k+1)-Y_n(k)=-\frac{\gamma\zeta(k/n^2) V'(Y_n(k))+\kappa}{n^2}+\frac{Z_{k+1}}{n},\quad k=0,1,\ldots,
\end{equation}
where $Z_1,Z_2,\ldots$ are independent standard normal random variables and $\zeta$ is the function defined in \eqref{zeta(t)}.

\begin{theorem}[{\bf Simulation of the tilted flashing Brownian ratchet}]
\textit{Let $\gamma>0$, $\kappa$ be real, and $\tau_1$, $\tau_2$, and $m$ be as above.  For $n=m,2m,3m,\ldots$, define the continuous-state Markov chain $\{Y_n(k),\, k=0,1,\ldots\}$ by \eqref{SDE2-discrete} $($given the initial state $Y_n(0)$, independent of $Z_1,Z_2,\ldots)$, and let $\{Y_t,\, t\ge0\}$ denote the tilted flashing Brownian ratchet with parameters $\gamma$, $\kappa$, $\tau_1$, and $\tau_2$.  If $Y_n(0)$ converges in distribution to $Y_0$ as $n\to\infty$, then $\{Y_n(\lfloor n^2 t\rfloor),\,t\ge0\}$ converges in distribution in $D_{\bf R}[0,\infty)$ to $\{Y_t,\, t\ge0\}$ as $n\to\infty$.  $($Here $n\to\infty$ through multiples of $m$.$)$}
\end{theorem}

Suppose, for example, we want to simulate the distribution of $Y_{\tau_1+\tau_2}$.  We choose an $n$, say $n=100$, and simulate $Y_n(n^2(\tau_1+\tau_2))$ by computing \eqref{SDE2-discrete} for $k=0,1,\ldots,n^2(\tau_1+\tau_2)-1$.  Repeating this computation a large number of times (e.g., $10^6$) will yield a histogram that approximates the distribution of interest.

Simulations of the flashing Brownian ratchet (albeit modified slightly from the one defined here) can be found in several sources, including Kinderlehrer and Kowalczyk~\cite[Fig.~2]{KK02} and H\"anggi and Marchesoni~\cite[Fig.~4]{HM09}.

\section{Numerical solution of the FP equation}\label{FP}

The Fokker--Planck equation (or Kolmogorov forward equation) is the equation for the transition density $p(t,x,y)$ of the diffusion process with diffusion coefficient 1 and drift coefficient $\mu$, namely
$$
\frac{\partial}{\partial t}p(t,x,y)=\frac12\frac{\partial^2}{\partial y^2}p(t,x,y)-\frac{\partial}{\partial y}[\mu(y)p(t,x,y)].
$$
We can solve it numerically using the finite-difference method.  We take $t_k:=k/n^2$ and $y_j:=j/n$, hence $\Delta t:=t_{k+1}-t_k=1/n^2$ and $\Delta y:=y_{j+1}-y_j=1/n$.  We write $p(t,y)$ for $p(t,x,y)$ (the initial state $x$ is fixed) to get
\begin{align*}
\frac{p(t_{k+1},y_j)-p(t_k,y_j)}{\Delta t}&=\frac12\,\frac{p(t_k,y_{j+1})-2p(t_k,y_j)+p(t_k,y_{j-1})}{(\Delta y)^2}\\
&\qquad{}-\frac{\mu(y_{j+1})p(t_k,y_{j+1})-\mu(y_{j-1})p(t_k,y_{j-1})}{2\Delta y}.
\end{align*}
(The increment in the last term is taken over $[y_{j-1},y_{j+1}]$ instead of $[y_j,y_{j+1}]$ for the purpose of symmetry.) 
This reduces to
\begin{align}\label{FP-tBr}
p(t_{k+1},y_j)&=p(t_k,y_j)+\frac12[p(t_k,y_{j+1})-2p(t_k,y_j)+p(t_k,y_{j-1})]\nonumber\\
&\qquad\qquad\quad{}-\frac{1}{2n}[\mu(y_{j+1})p(t_k,y_{j+1})-\mu(y_{j-1})p(t_k,y_{j-1})]\\
&=\frac12\,p(t_k,y_{j+1})\bigg(1-\frac{1}{n}\mu(y_{j+1})\bigg)+\frac12\,p(t_k,y_{j-1})\bigg(1+\frac1n\mu(y_{j-1})\bigg),\nonumber
\end{align}
which we can solve recursively, once we specify the various $p(0,y_j)$.

Here we expect convergence as $n\to\infty$, but we do not have a theorem to this effect.  Perhaps the results for the heat equation in Mitchell and Griffiths~\cite[Section~2.6]{MG80} can be generalized to this setting.

For the corresponding density $q(t,y)$ of the tilted flashing Brownian ratchet, \eqref{FP-tBr} is replaced by
\begin{align}\label{FP-tfBr}
q(t_{k+1},y_j)&=\frac12\,q(t_k,y_{j+1})\bigg(1-\frac1n\zeta(t_k)\mu(y_{j+1})+\frac1n(1-\zeta(t_k))\kappa \bigg)\nonumber\\
&\qquad{}+\frac12\,q(t_k,y_{j-1})\bigg(1+\frac{1}{n}\zeta(t_k)\mu(y_{j-1})-\frac1n(1-\zeta(t_k))\kappa\bigg),
\end{align}
where $\zeta$ is the function defined in \eqref{zeta(t)}.

As we will see in the next section, there is a close relationship between \eqref{FP-tfBr} and Theorem~2$'$ (and between \eqref{FP-tBr} and Theorem~1$'$).

\section{Density at time $\tau_1+\tau_2$, starting at 0}\label{modeling1}

To model Figure~\ref{HAT00fig}, we make some measurements and find that, as explained in \cite{EL18}, $\alpha=1/4$ (hence $l=1$ and $L=4$) and $\tau_1=2.4$.  We cannot estimate $\tau_2$ from the figure, so for convenience, we take $\tau_2=\tau_1$.

We surmise that the tilted flashing Brownian ratchet described in Figure~\ref{HAT00fig} evolves as a Brownian motion with drift $-\kappa$, starting at 0, for time $\tau_1=2.4$. Then it evolves as the tilted Brownian ratchet with  $\alpha=1/4$, $L=4$, $\gamma$, and $\kappa$, starting from where the Brownian motion ended, for time $\tau_2=2.4$.  We seek the distribution of the tilted flashing Brownian ratchet at time $\tau_1+\tau_2=4.8$, which we can compare with panels ($c$) and ($f$) of Figure~\ref{HAT00fig}.

No analytical formula for the density of the tilted flashing Brownian ratchet at a fixed time is available.  Nevertheless, it can be approximated numerically, as Theorem~2$'$ suggests.  More precisely, we approximate $Y_{\tau_1+\tau_2}$ by $Y_n(n^2(\tau_1+\tau_2))/n$.  The number $m$ of that theorem is 5.  In each case we choose $n=100$, implying that at time $\tau_1+\tau_2$, the approximating random walk has made $n^2(\tau_1+\tau_2)=48{,}000$ steps.  Starting from $P_0(0)=1$, we compute its distribution $P_k(j):=\P(Y_n(k)=j)$ recursively after 1 step, 2 steps, \dots, $n^2(\tau_1+\tau_2)$ steps, using the simple random walk with \eqref{p} and \eqref{rho,eps} for the first $n^2 \tau_1$ steps, 
\begin{equation}\label{rw-recursion1}
P_{k+1}(j)=P_k(j+1)(1-p)+P_k(j-1)p,
\end{equation}
then the random walk of Theorem~1$'$ for the next $n^2\tau_2$ steps,
\begin{equation}\label{rw-recursion2}
P_{k+1}(j)=P_k(j+1)(1-p_{I(j+1)})+P_k(j-1)p_{I(j-1)},
\end{equation}
where 
$$
I(j):=\begin{cases}0&\text{if $\text{mod}(j,nL)<nl$,}\\1&\text{if $\text{mod}(j,nL)\ge nl$.}\end{cases}
$$

This allows us to plot the histogram of $Y_n(n^2(\tau_1+\tau_2))/n$, interpolating linearly to approximate the density of $Y_{\tau_1+\tau_2}$.  In Figure~\ref{tfBr0-1-5}, we display the results for $\gamma=\lambda(1-\alpha)/2$ with $\lambda=1,2,3,4,5$ and $\kappa=\theta\kappa_0/2$ with $\theta=-1, 0,1,2,3,4$, where $\kappa_0=0.2748$ is the value of $\kappa$, accurate to four significant digits, for which mean displacement is 0 when $\lambda=5$ and $n=100$.

\begin{figure}[ht]
\centering
\includegraphics[width=1.5in]{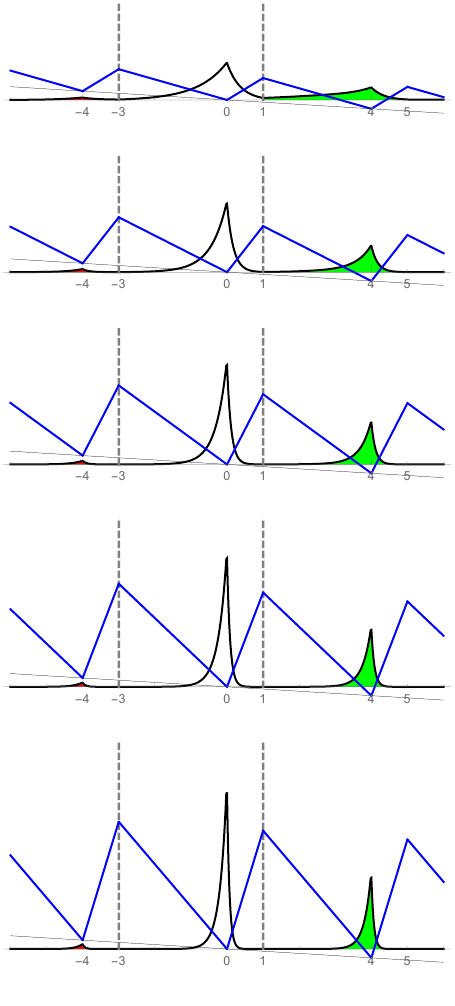}
\includegraphics[width=1.5in]{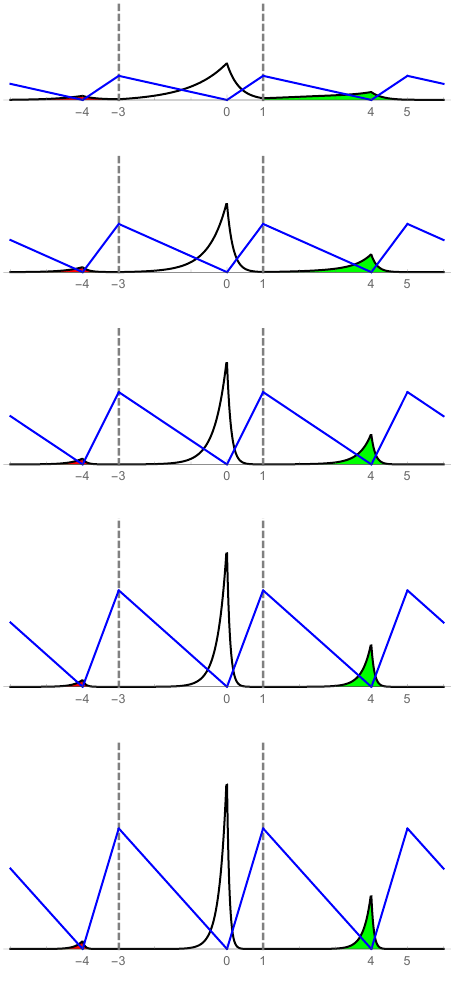}
\includegraphics[width=1.5in]{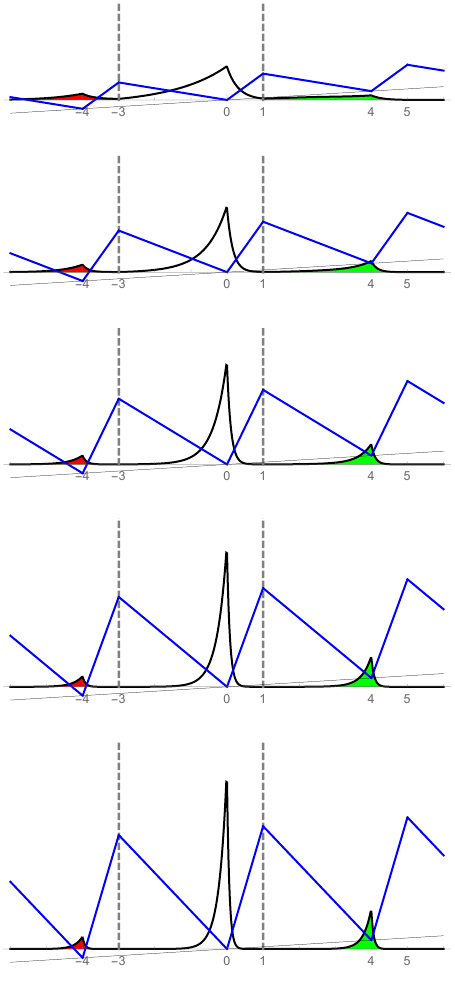}\\
{\small ($a$) $\kappa=-\kappa_0/2$  \hspace{1.8cm} ($b$) $\kappa=0$ \hspace{2.2cm}   ($c$) $\kappa=\kappa_0/2$ }
\caption{\label{tfBr0-1-5}A Brownian motion with drift $-\kappa$, starting at 0, runs for time $\tau_1=2.4$.  Then, starting from where the Brownian motion ended, a tilted Brownian ratchet with $\alpha=1/4$, $L=4$, $\gamma=\lambda(1-\alpha)/2$, $\lambda=1,2,3,4,5$ (from top to bottom), and $\kappa=\theta\kappa_0/2$, $\kappa_0=0.2748$, $\theta=-1,0,1,2,3,4$ (from left to right), runs for time $\tau_2=2.4$.  The black curve is an approximation to the density of the tilted flashing Brownian ratchet at time $\tau_1+\tau_2$, via the improved random walk approximation with $n=100$. The blue curve is the sawtooth potential $\gamma V$ (not just $V$ as in \cite{EL18}), tilted with slope $\kappa$, then shrunk vertically by a factor of $L$ for clarity.  (See the paragraph below \eqref{rw-recursion2} for the meaning of $\kappa_0$.)}
\end{figure}

\setcounter{figure}{1}

\begin{figure}[th]
\centering
\includegraphics[width=1.5in]{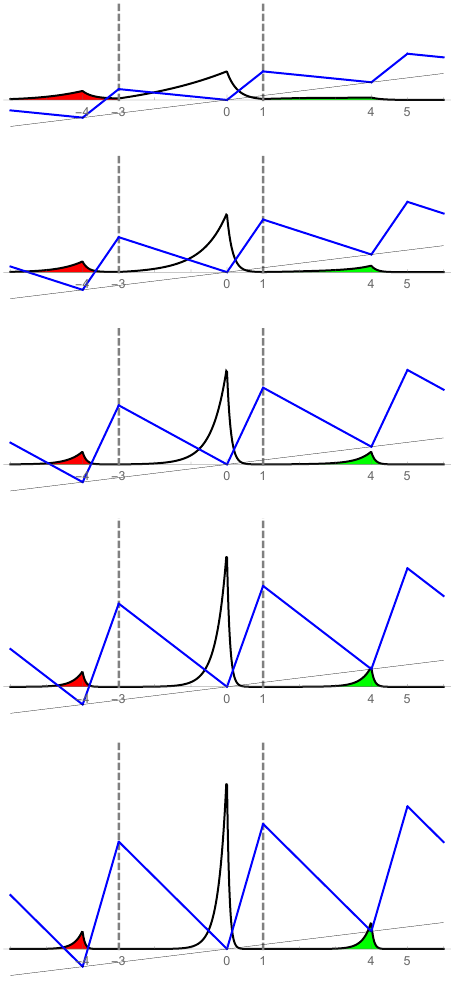}
\includegraphics[width=1.5in]{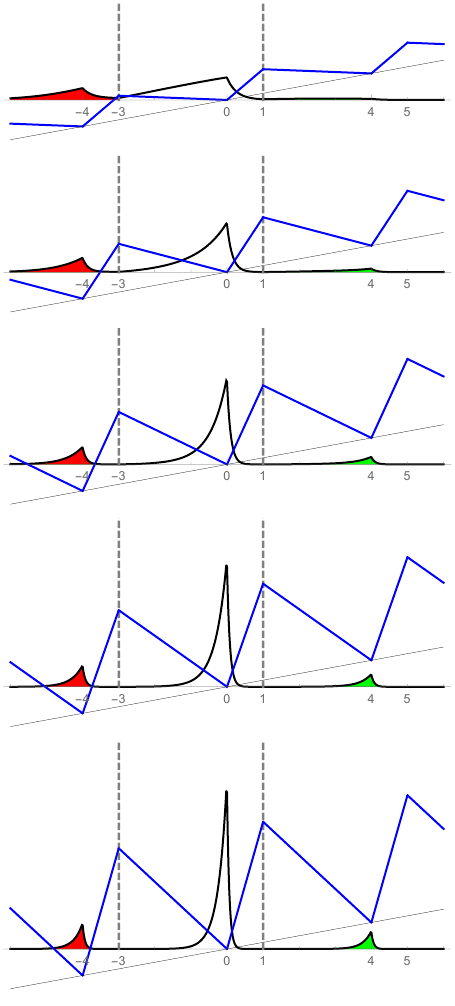}
\includegraphics[width=1.5in]{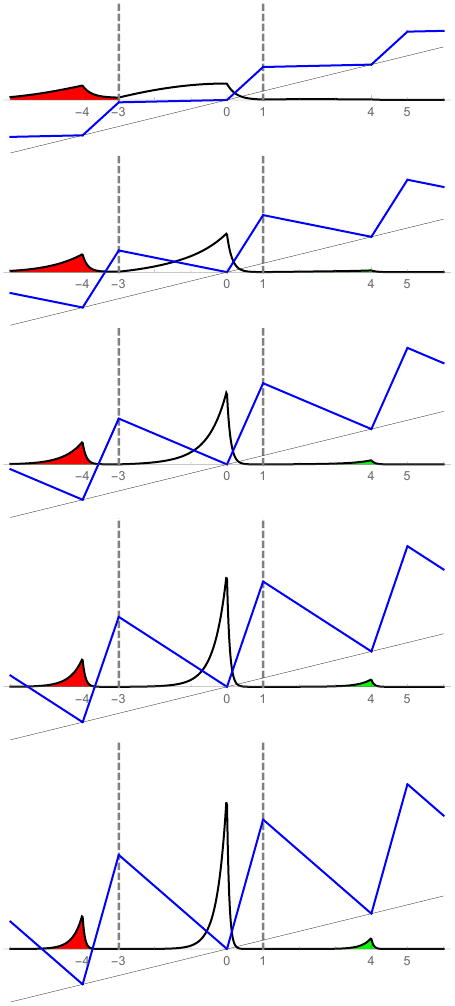}\\
{\small ($d$) $\kappa=\kappa_0$ \hspace{2cm} ($e$) $\kappa=3\kappa_0/2$ \hspace{2cm} ($f$) $\kappa=2\kappa_0$}
\caption{\label{tfBr0-1-5a}(Continued) A Brownian motion with drift $-\kappa$, starting at 0, runs for time $\tau_1=2.4$.  Then, starting from where the Brownian motion ended, a tilted Brownian ratchet with $\alpha=1/4$, $L=4$, $\gamma=\lambda(1-\alpha)/2$, $\lambda=1,2,3,4,5$ (from top to bottom), and $\kappa=\theta\kappa_0/2$, $\kappa_0=0.2748$, $\theta=-1,0,1,2,3,4$ (from left to right), runs for time $\tau_2=2.4$.  The black curve is an approximation to the density of the tilted flashing Brownian ratchet at time $\tau_1+\tau_2$, via the improved random walk approximation with $n=100$. The blue curve is the sawtooth potential $\gamma V$ (not just $V$ as in \cite{EL18}), tilted with slope $\kappa$, then shrunk vertically by a factor of $L$ for clarity.}
\end{figure}

Before commenting on these figures, we observe that they could also have been derived using the methods of Sections \ref{approx} (Theorem~\ref{theorem2}), \ref{EM} (Theorem~4), or \ref{FP}. However, stochastic simulation (Section~\ref{EM}) provides slightly less smooth figures in a comparable amount of time.  For numerical solution of the Fokker--Planck equation (Section~\ref{FP}), recursion \eqref{rw-recursion1} and \eqref{rw-recursion2} with $p_0$ and $p_1$ as in \eqref{p0,p1-improved} is identical with recursion \eqref{FP-tfBr} when $\gamma = \lambda (1-\alpha)/2$.  Thus, the method of Section~\ref{FP} gives results identical to the improved random walk approximation of Section~\ref{approx} (Theorem~2$'$).  

This observation is not entirely new.  The main conclusion of Allison and Abbott~\cite{AA02} as well as Toral et al.~\cite{TAM03a,TAM03b} is that, if the Fokker--Planck equation for the Brownian ratchet is suitably discretized, one gets an equation that is identical to the master equation for Parrondo's capital-dependent game $B$, at least for a suitable potential.  A similar result was found by Heath et al.~\cite{HKK02}.  What \textit{is} new in our approach is that we recover the specific potential in the definition of the Brownian ratchet.  This requires not just letting $\Delta t$ and $\Delta y$ in the finite-difference method go to zero with $(\Delta y)^2/\Delta t$ constant, but also requires letting the probabilities in the Parrondo games go to $1/2$ in a very specific way, unlike in the earlier works where these probabilities were fixed.  

We computed statistics for the $n$th random walk approximation (improved and unimproved) of the tilted flashing Brownian ratchet with $\alpha=1/4$, $L=4$, $\gamma=\lambda(1-\alpha)/2$, $\lambda=5$, $\kappa=\kappa_0=0.2748$, $\tau_1=\tau_2=2.4$, and initial state 0, at time $\tau_1+\tau_2$, for $n=10,20,30,\ldots,200$.  The rate of convergence using the improved approximation appeared notably faster.  We omit these tables to save space.

Although the conceptual figures in panels ($c$) and ($f$) of Figure~\ref{HAT00fig} are largely accurate, the figures of Figure~\ref{tfBr0-1-5} offer several key improvements, just as we found in the untilted case.  First, the three peaks of the density are pointed, unlike a normal density.  Second, they are asymmetric, again unlike a normal density, with more mass to the left than to the right of $-4$, 0, and 4.  This is largely due to the fact that, for example, the drift to the left on $(0, 1)$ is stronger than the drift to the right on $(-3, 0)$, but there also seems to be some dependence on the slope of the tilt.  Third, the relative heights of the peaks vary more than in the conceptual figure.  The simulations cited previously had implicitly noted the pointed \cite{KK02} and asymmetric \cite{KK02,HM09} peaks.

The peaks in Figure~\ref{HAT00fig} appear to be normal, while the peaks in Astumian and H\"anggi~ \cite[Fig.~2]{AH02} are explicitly called ``gaussians."  This thinking led Astumian~\cite{A97} to conclude that mean displacement at time $\tau_1+\tau_2$ has the form
$$
\sum_{j=-\infty}^\infty jL\,P_{jL},\quad\text{where } P_{jL}:=\P(B_{\tau_1}-\kappa\tau_1\in((j-1+\alpha)L,(j+\alpha)L]).
$$
This is inaccurate for two reasons.  Most importantly, the mean of the peak at $jL$ is not $jL$ because of the skewed nature of that peak.  Also, the area of the peak at $jL$ is not the normal probability $P_{jL}$; this is only an approximation that becomes exact in the limit as $\gamma\to\infty$.   

We consider a statistic that measures the skewness of the peaks of the distribution.  The skewness can be defined to be the mass of the distribution in $\bigcup_{n=-\infty}^\infty (nL,nL+\alpha L)$ minus the mass of the distribution in $\bigcup_{n=-\infty}^\infty (nL-(1-\alpha)L,nL)$.  A negative skewness means that the peaks are skewed to the left.

Table~\ref{computations-theta} shows the effect of varying $\kappa$, when $\lambda=5$, on several statistics of interest, namely areas and heights of the three peaks, mean displacement, and skewness.  Figure~\ref{mean-contour} shows that the mean displacement for the $n$th improved random walk ($n=100$) approximating the tilted flashing Brownian ratchet, starting at 0, is increasing in $\gamma$ and decreasing in $\kappa$. In fact, it decreases nearly linearly in $\kappa$.

\begin{table}[thb]
\caption{\label{computations-theta} Computations for the $n$th improved random walk ($n=100$) approximating the tilted flashing Brownian ratchet with $\alpha=1/4$, $L=4$, $\gamma=\lambda(1-\alpha)/2$, $\lambda=5$, $\kappa=\theta\kappa_0/2$ for $\kappa_0=0.2748$ and various $\theta$, $\tau_1=\tau_2=2.4$, and initial state 0, at time $\tau_1+\tau_2$, illustrating the effect of varying the slope of the tilt.  (See the paragraph below \eqref{rw-recursion2} for the meaning of $\kappa_0$.) }
\tabcolsep=.32cm
\vspace{-0.5cm}
\begin{center}
\begin{footnotesize}
{\begin{tabular}{@{}ccccc@{}}
\noalign{\smallskip}
\hline
\noalign{\smallskip}
$\theta$   &  areas of the three peaks   &heights of the three peaks & mean & skewness \\
& & [at $-4,0,4$]& displace. & \\
\noalign{\smallskip}
\hline
\noalign{\smallskip}
            $-1.5$ &  $(0.01521,0.6381,0.3466)$  &  $(0.05597,2.348,1.269)\hphantom{0}$ &    \hphantom{$-$}1.210\hphantom{00}  & $-0.4930$\\
            $-1.0$ &  $(0.01995,0.6727,0.3074)$  &  $(0.07215,2.433,1.107)\hphantom{0}$ &    \hphantom{$-$}1.027\hphantom{00}  & $-0.5062$\\
            $-0.5$ &  $(0.02588,0.7040,0.2701)$  &  $(0.09197,2.501,0.9563)$            &    \hphantom{$-$}0.8474\hphantom{0}  & $-0.5194$\\
\hphantom{$-$}0.0  &  $(0.03324,0.7314,0.2353)$  &  $(0.1159,2.551,0.8178)\hphantom{0}$ &    \hphantom{$-$}0.6716\hphantom{0}  & $-0.5325$\\
\hphantom{$-$}0.5  &  $(0.04225,0.7547,0.2031)$  &  $(0.1445,2.581,0.6921)\hphantom{0}$ &    \hphantom{$-$}0.4995\hphantom{0}  & $-0.5457$\\
\hphantom{$-$}1.0  &  $(0.05317,0.7732,0.1736)$  &  $(0.1782,2.591,0.5796)\hphantom{0}$ &    \hphantom{$-$}0.3308\hphantom{0}  & $-0.5589$\\
\hphantom{$-$}1.5  &  $(0.06622,0.7868,0.1470)$  &  $(0.2173,2.580,0.4801)\hphantom{0}$ &    \hphantom{$-$}0.1646\hphantom{0}  & $-0.5721$\\
\hphantom{$-$}2.0  &  $(0.08166,0.7951,0.1233)$  &  $(0.2620,2.549,0.3933)\hphantom{0}$ &    \hphantom{$-$}0.0000$^{\rm 2}$           & $-0.5853$\\
\hphantom{$-$}2.5  &  $(0.09970,0.7980,0.1023)$  &  $(0.3125,2.499,0.3185)\hphantom{0}$ &                $-0.1643$\hphantom{0} & $-0.5985$\\
\hphantom{$-$}3.0  &  $(0.1205,0.7954,0.08405)$  &  $(0.3687,2.429,0.2549)\hphantom{0}$ &                $-0.3296$\hphantom{0} & $-0.6117$\\
\hphantom{$-$}3.5  &  $(0.1443,0.7873,0.06834)$  &  $(0.4302,2.342,0.2016)\hphantom{0}$ &                $-0.4971$\hphantom{0} & $-0.6250$\\
\hphantom{$-$}4.0  &  $(0.1711,0.7739,0.05498)$  &  $(0.4965,2.238,0.1574)\hphantom{0}$ &                $-0.6682$\hphantom{0} & $-0.6383$\\
\hphantom{$-$}4.5  &  $(0.2010,0.7553,0.04376)$  &  $(0.5668,2.121,0.1214)\hphantom{0}$ &                $-0.8439$\hphantom{0} & $-0.6516$\\
\noalign{\smallskip}
\hline
\noalign{\smallskip}
&&& \multicolumn{2}{r}{$^{\rm 2}$\,$-0.000001280$}\\
\end{tabular}}    
\end{footnotesize} 
\end{center}        
\end{table}

\begin{figure}[!h]
\centering
\includegraphics[width=3.5in]{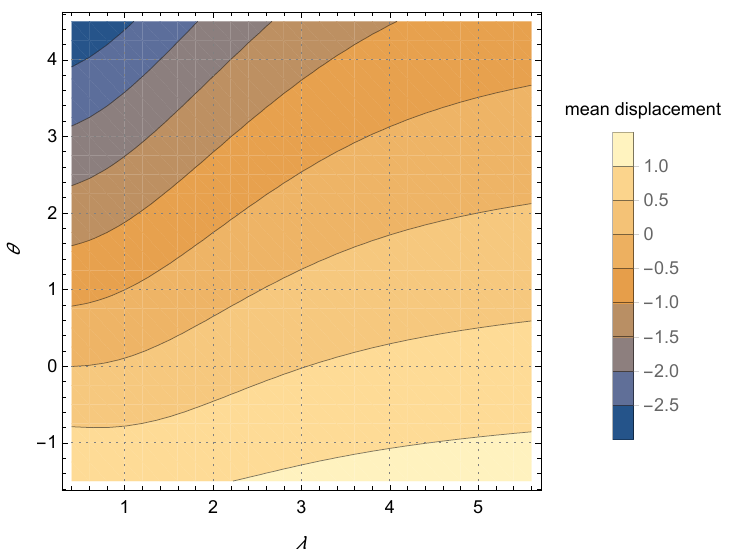}
\caption{\label{mean-contour}Mean displacement for the $n$th improved random walk ($n=100$) approximating the tilted flashing Brownian ratchet with $\alpha=1/4$, $L=4$, $\gamma=\lambda(1-\alpha)/2$, $\kappa=\theta\kappa_0/2$, $\kappa_0=0.2748$, $\tau_1=\tau_2=2.4$, and initial state 0, at time $\tau_1+\tau_2$, as a function of $\lambda$ and $\theta$.}
\end{figure}

In particular, let us define $\kappa_0(\lambda)$ to be such that the contour with mean displacement 0 passes through $(\lambda,\kappa_0(\lambda))$.  For example, $\kappa_0(5)=0.2748$ to four significant digits.  Figure~\ref{mean-contour} shows that $\kappa_0(\lambda)$ increases in $\lambda$.  It appears that $\kappa_0(\lambda)\to5/12$ as $\lambda\to\infty$.  More generally, 
$$
\kappa_0(\lambda)\to\frac{\big(\frac12-\alpha\big)L}{\tau_1}\quad\text{as}\quad \lambda\to\infty.
$$
Indeed, we expect that $\lim_{\lambda\to\infty}\kappa_0(\lambda)$ is the value of $\kappa$ that makes $\P(Y_{\tau_1}\ge\alpha L)=\P(Y_{\tau_1}<-(1-\alpha)L)$.  But $Y_{\tau_1}=B_{\tau_1}-\kappa\tau_1$, so $\kappa$ must satisfy $\kappa\tau_1+\alpha L=-[\kappa\tau_1-(1-\alpha)L]$, and the conclusion follows.  This result is due to Astumian~\cite{A97}.

Figure~\ref{skewness} (also Table~\ref{computations-theta}) shows that the absolute value of the skewness for the $n$th improved random walk ($n=100$) approximating the tilted flashing Brownian ratchet, starting at 0, is increasing in $\kappa$.

\begin{figure}[!h]
\centering
\includegraphics[width=3.2in]{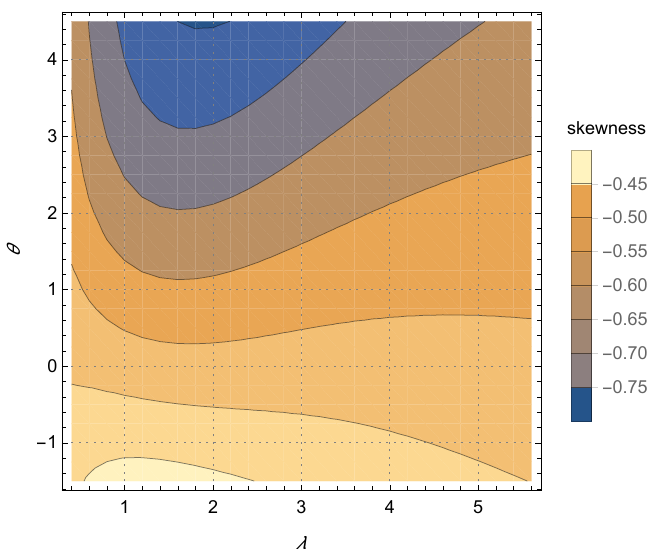}
\caption{\label{skewness}Skewness for the $n$th improved random walk ($n=100$) approximating the tilted flashing Brownian ratchet with $\alpha=1/4$, $L=4$, $\gamma=\lambda(1-\alpha)/2$, $\kappa=\theta\kappa_0/2$, $\kappa_0=0.2748$, $\tau_1=\tau_2=2.4$, and initial state 0, at time $\tau_1+\tau_2$, as a function of $\lambda$ and $\theta$.}
\end{figure}

\section{Density at time $\tau_1+\tau_2$, starting at stationarity}\label{modeling2}

We may want to wrap the tilted Brownian ratchet and the tilted flashing Brownian ratchet around the circle of circumference $L$.  By virtue of the fact that their behavior is spatially periodic with period $L$, the wrapped processes remain Markovian.  We define the \textit{wrapped tilted Brownian ratchet} as the $[0,L)$-valued process
$$
\bar X_t:=\text{mod}(X_t,L),
$$
where $X_t$ denotes the tilted Brownian ratchet in \eqref{SDE1}, with the understanding that the endpoints of the interval $[0,L)$ are identified, making it a circle of circumference $L$.  The same procedure yields the \textit{wrapped tilted flashing Brownian ratchet}, 
$$
\bar Y_t:=\text{mod}(Y_t,L),
$$
where $Y_t$ denotes the tilted flashing Brownian ratchet in \eqref{SDE2}.  Note that the skewness, as it was defined in the preceding section, can be determined for the original processes from the wrapped processes.  The same is not true for the mean displacement.

We begin by finding the stationary distribution of the wrapped tilted Brownian ratchet.  This process is a diffusion process in $[0,L]$ (endpoints are identified) with generator
$$
(\mathscr{L}f)(x):=\frac12 f''(x)+\mu(x)f'(x).
$$
To begin, let us assume that $\mu$ is a smooth function with $\mu(0)=\mu(L)$.  We hope that, once we have the desired formula, we can substitute the actual $\mu$, namely \eqref{mu(x)}, which has discontinuities.

We want to find the function $\phi$ satisfying
$$
\int_0^L (\mathscr{L}f)(x)\phi(x)\,dx=0,
$$
$\phi(0)=\phi(L)$, and $\phi'(0)=\phi'(L)$ for all smooth $f$ with $f(0)=f(L)$ and $f'(0)=f'(L)$.  We expect that
\begin{equation}\label{adj}
\int_0^L (\mathscr{L}f)(x)\phi(x)\,dx=\int_0^L f(x)(\mathscr{L}^*\phi)(x)\,dx,
\end{equation}
where
$$
(\mathscr{L}^*\phi)(x):=\frac12 \phi''(x)-(\mu \phi)'(x)
$$
is the formal adjoint of $\mathscr{L}$, showing that it suffices to solve the differential equation
\begin{equation}\label{ode}
(\mathscr{L}^*\phi)(x)=0.
\end{equation}
The justification of \eqref{adj} requires three integrations by parts and $f(0)=f(L)$, $f'(0)=f'(L)$, $\phi(0)=\phi(L)$, $\phi'(0)=\phi'(L)$, and $\mu(0)=\mu(L)$. 

It remains to solve the differential equation \eqref{ode}, which is equivalent to
$$
\frac12 \phi'(x)-\mu(x)\phi(x)=C, \qquad \phi(0)=\phi(L),\quad \phi'(0)=\phi'(L),
$$
for some constant $C$.  Letting
$$
M(x):=\int_0^x \mu(y)\,dy,
$$
we find using elementary methods that
\begin{equation}\label{stationary}
\phi(x)=\phi(0)e^{2M(x)} \bigg(1-\big[1-e^{-2M(L)}\big]\frac{\int_0^x e^{-2M(y)}\,dy}{\int_0^L e^{-2M(y)}\,dy}\bigg).
\end{equation}
The solution is unique, once we require that
$$
\int_0^L \phi(x)\,dx=1.
$$

Now we can use this result with the desired $\mu$, even though $\mu$ has discontinuities.  Specifically, with
\begin{equation}\label{M(x)}
M(x)=-[\gamma V(x)+\kappa x],
\end{equation}
we have an analytical formula for $\phi$, which can be shown to be equivalent to a formula of Reimann~\cite[Eq.~(2.36)]{R02}.

\begin{theorem}[{\bf Stationary distribution of the wrapped tilted Brownian ratchet}]
\textit{The wrapped tilted Brownian ratchet with parameters $\alpha$, $L$, $\gamma$, and $\kappa$ has a stationary distribution $\pi$ of the form \eqref{stationary} $($using \eqref{M(x)}$)$, restricted to $[0,L]$.}
\end{theorem}

The skewness of this stationary distribution can be evaluated.  In the special case $\kappa=0$, the computation is especially easy and we find that the skewness is $2\alpha-1$; surprisingly, it does not depend on $\gamma$.  

We denote the tilted flashing Brownian ratchet at time $t$, starting from $x\in{\bf R}$ at time 0, by $Y_t^x$, and the wrapped tilted flashing Brownian ratchet at time $t$, starting from $x\in[0,L)$ at time 0, by $\bar Y_t^x$.  Then the one-step transition function
\begin{equation}\label{transition-function}
\bar P(x,\cdot):=\P(\bar Y_{\tau_1+\tau_2}^x\in\cdot)
\end{equation}
for a continuous-state Markov chain has a stationary distribution $\bar\pi$. The mean displacement $\bar\mu$ of the tilted flashing Brownian ratchet over the time interval $[0,\tau_1+\tau_2]$, starting from the stationary distribution $\bar\pi$, namely
\begin{align*}
\bar\mu&:=\int_0^L\E[Y_{\tau_1+\tau_2}^x-Y_0^x]\,\bar\pi(dx)\\
&\phantom{:}=\int_{-(1-\alpha)L}^{\alpha L}\E[Y_{\tau_1+\tau_2}^x-Y_0^x]\,\bar\pi(dx)
\end{align*}
is of interest. The second equality is based on the periodicity of the integrand (with period $L$) and the convention that we do not distinguish notationally between $\bar\pi$ and its image under the mapping
$$
x\mapsto\begin{cases}x&\text{if $0\le x<\alpha L$,}\\ x-L&\text{if $\alpha L\le x< L$.}\end{cases}
$$

We approximate $\bar\mu$ as follows.  The integrand can be estimated as in Section~\ref{modeling1}, the only distinction being that the starting point of the tilted flashing Brownian ratchet is $x$, not 0.  The stationary distribution $\bar\pi$ of the one-step transition function~\eqref{transition-function} can be approximated by the stationary distribution of the finite Markov chain whose one-step transition matrix has the form 
$$
P(i,j):=\P(\bar Y_n(n^2(\tau_1+\tau_2))=j\mid \bar Y_n(0)=i),\quad i,j=0,1,\ldots,nL-1,
$$
where $\{\bar Y_n(k),\, k=0,1,\ldots\}$ denotes the wrapped (period-$nL$) random walk used to approximate the wrapped tilted flashing Brownian ratchet.  A minor technical issue, if $nL$ is even,  is that this Markov chain fails to be irreducible if $n^2(\tau_1+\tau_2)$ is even, in which case we replace $n^2(\tau_1+\tau_2)$ by $n^2(\tau_1+\tau_2)+1$.  Then the chain is irreducible and there is a unique stationary distribution.  The black curve in the first row of the panels in Figure~\ref{tfBrstat-fig} approximates the density of $\bar\pi$ with support $[-3,1)$.  Starting from the approximate $\bar\pi$ at time 0, the black curves in the second and third rows of each panel are the approximations to the density of the tilted flashing Brownian ratchet at times $\tau_1$ and $\tau_1+\tau_2$, respectively. When $\lambda=5$ and $\kappa=\kappa_0=0.2748$, computations show that $\bar\mu=0.01519$, which is slightly larger than the corresponding number in Table~\ref{computations-theta}, namely 0.0000.  A more complete set of such computations is given as Table~\ref{computations-theta-stationarity}.  It is interesting to compare this table (starting at stationarity) with Table~\ref{computations-theta} (starting at 0).

\begin{figure}[p]
\centering
\includegraphics[width=1.5in]{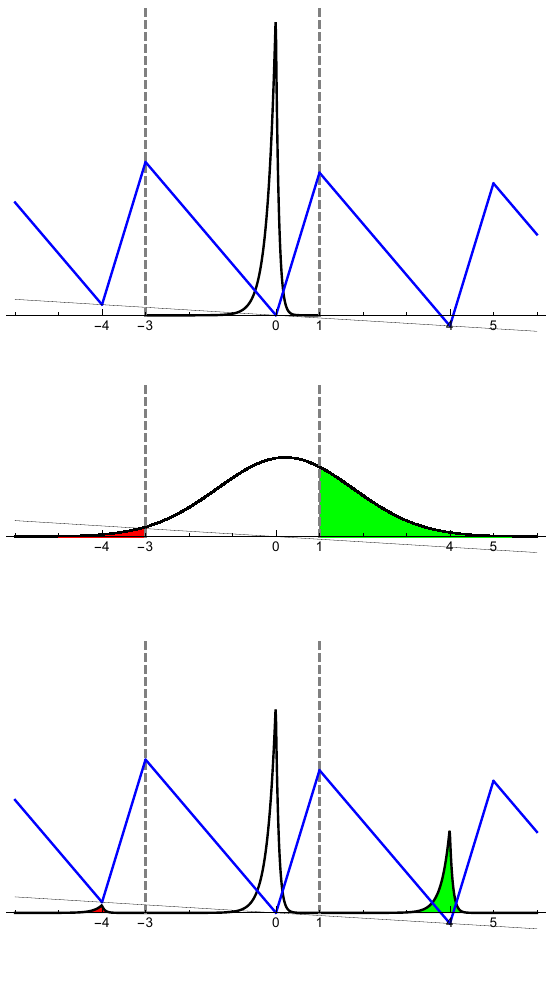}
\includegraphics[width=1.5in]{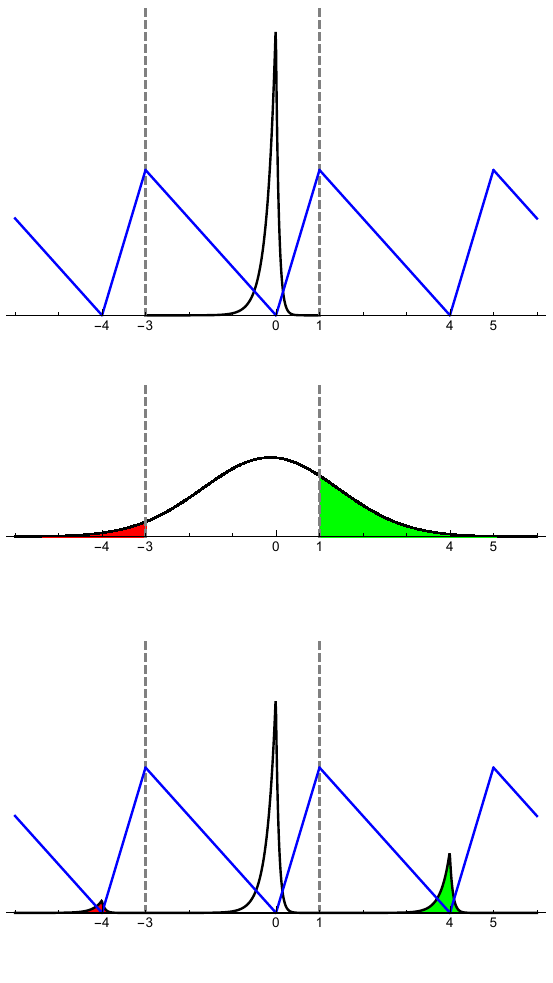}
\includegraphics[width=1.5in]{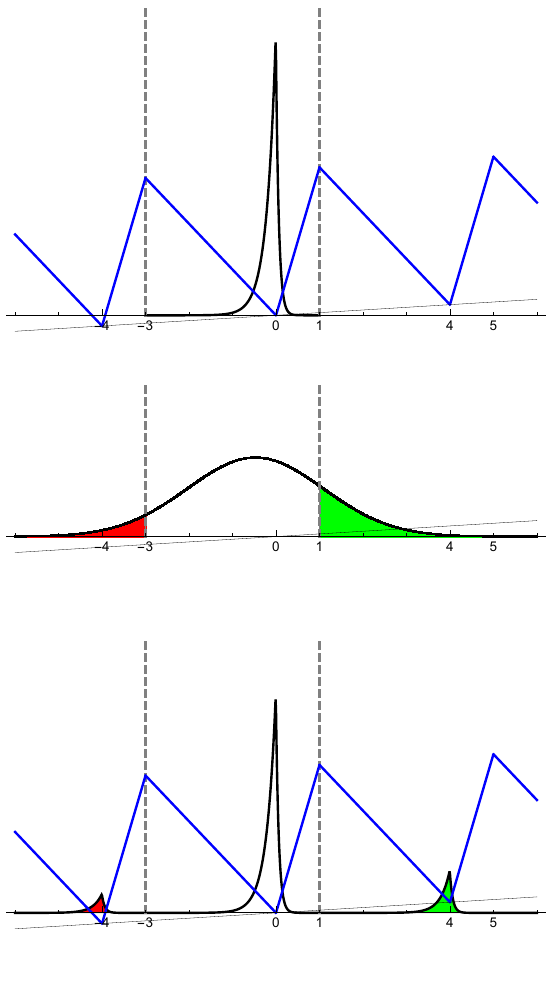}\\
{\small ($a$) $\kappa=-\kappa_0/2$ \hspace{1.8cm} ($b$) $\kappa=0$ \hspace{2.2cm} ($c$) $\kappa=\kappa_0/2$}
\\
\vspace{0.5cm}
\includegraphics[width=1.5in]{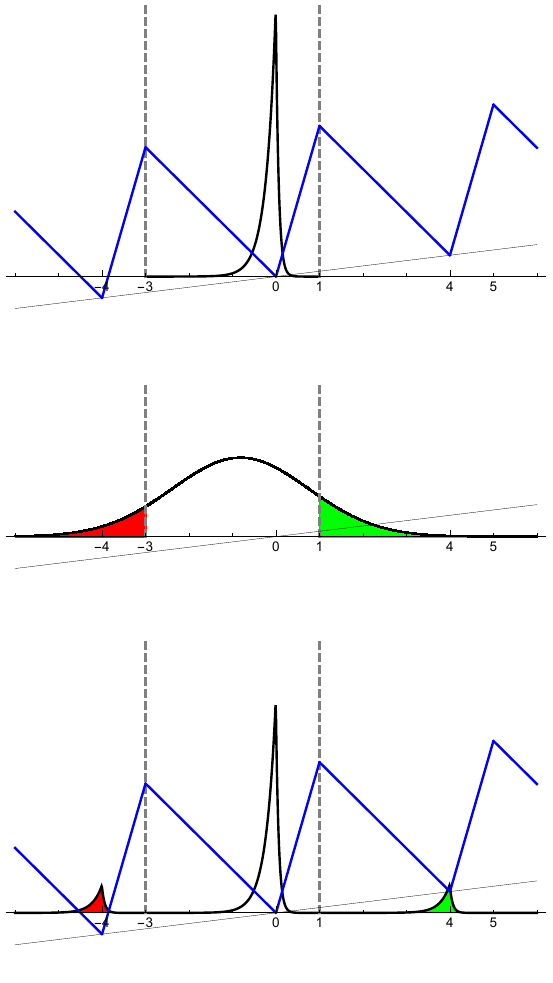}
\includegraphics[width=1.5in]{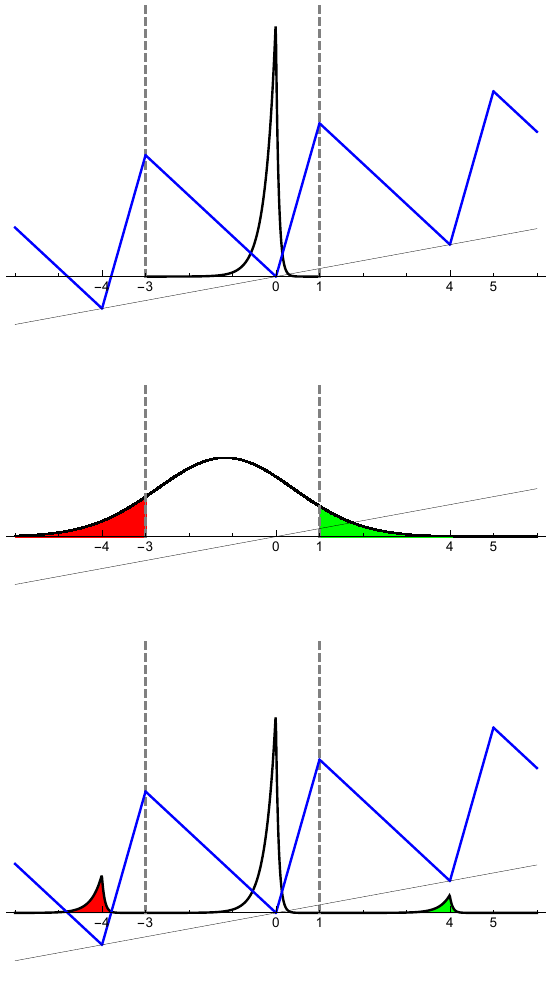}
\includegraphics[width=1.5in]{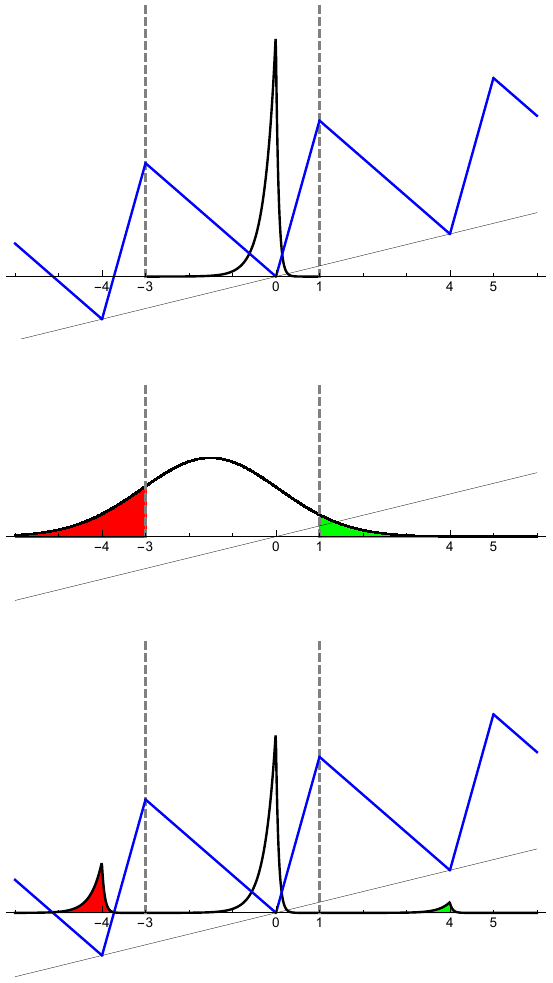}\\
{\small ($d$) $\kappa=\kappa_0$ \hspace{2cm} ($e$) $\kappa=3\kappa_0/2$ \hspace{2cm} ($f$) $\kappa=2\kappa_0$}
\caption{\label{tfBrstat-fig}Starting from the stationary distribution $\bar\pi$ with support $[-3,1)$, a Brownian motion with drift $-\kappa$ runs for time $\tau_1=2.4$.  Then, starting from where the Brownian motion ended, a tilted Brownian ratchet with $\alpha=1/4$, $L=4$, $\gamma=\lambda(1-\alpha)/2$, $\lambda=5$, and $\kappa=\theta\kappa_0/2$, $\kappa_0=0.2748$ (with $\theta=-1,0,1,2,3,4$), runs for time $\tau_2=2.4$.  The black curves approximate the density of the tilted flashing Brownian ratchet at times 0, $\tau_1$, and $\tau_2+\tau_2$, via the improved random walk approximation with $n=100$.  The blue curves are as in Figure~\ref{tfBr0-1-5}.  The vertical axes in the panels of the first and third rows are comparable, whereas the vertical axis in the panels of the second row (for the normal-like curve but not for the tilted axis) has been stretched by a factor of $L$ for clarity.}
\end{figure}

\begin{table}[!htb]
\caption{\label{computations-theta-stationarity} Computations for the $n$th improved random walk ($n=100$) approximating the tilted flashing Brownian ratchet with $\alpha=1/4$, $L=4$, $\gamma=\lambda(1-\alpha)/2$, $\lambda=5$, $\kappa=\theta\kappa_0/2$ for $\kappa_0=0.2748$ and various $\theta$, and $\tau_1=\tau_2=2.4$, starting at stationarity, at time $\tau_1+\tau_2$, illustrating the effect of varying the slope of the tilt.  (See the paragraph below \eqref{rw-recursion2} for the meaning of $\kappa_0$.) }
\tabcolsep=.32cm
\vspace{-0.5cm}
\begin{center}
\begin{footnotesize}
{\begin{tabular}{@{}ccccc@{}}
\noalign{\smallskip}
\hline
\noalign{\smallskip}
$\theta$   &  areas of the three peaks   &heights of the three peaks & mean & skewness \\
&& [at $-4,0,4$]& displace. & \\
\noalign{\smallskip}
\hline
\noalign{\smallskip}
            $-1.5$ &  $(0.01929,0.6609,0.3199)$             &  $(0.07095,2.431,1.171)\hphantom{0}$ &   \hphantom{$-$}1.207\hphantom{00}  & $-0.4957$\\
            $-1.0$ &  $(0.02523,0.6936,0.2811)$             &  $(0.09127,2.509,1.013)\hphantom{0}$ &   \hphantom{$-$}1.027\hphantom{00}  & $-0.5089$\\
            $-0.5$ &  $(0.03267,0.7226,0.2448)$             &  $(0.1161,2.567,0.8666)\hphantom{0}$ &   \hphantom{$-$}0.8507\hphantom{0}  & $-0.5220$\\
\hphantom{$-$}0.0  &  $(0.04185,0.7471,0.2110)$             &  $(0.1459,2.605,0.7336)\hphantom{0}$ &   \hphantom{$-$}0.6783\hphantom{0}  & $-0.5352$\\
\hphantom{$-$}0.5  &  $(0.05304,0.7668,0.1802)$             &  $(0.1814,2.622,0.6141)\hphantom{0}$ &   \hphantom{$-$}0.5094\hphantom{0}  & $-0.5483$\\
\hphantom{$-$}1.0  &  $(0.06654,0.7812,0.1522)$             &  $(0.2230,2.618,0.5082)\hphantom{0}$ &   \hphantom{$-$}0.3433\hphantom{0}  & $-0.5615$\\
\hphantom{$-$}1.5  &  $(0.08260,0.7901,0.1273)$             &  $(0.2709,2.591,0.4157)\hphantom{0}$ &   \hphantom{$-$}0.1790\hphantom{0}  & $-0.5746$\\
\hphantom{$-$}2.0  &  $(0.1015,0.7932,0.1053)\hphantom{0}$  &  $(0.3255,2.543,0.3360)\hphantom{0}$ &   \hphantom{$-$}0.01519             & $-0.5878$\\
\hphantom{$-$}2.5  &  $(0.1234,0.7904,0.08614)$             &  $(0.3867,2.475,0.2682)\hphantom{0}$ &               $-0.1495$\hphantom{0} & $-0.6010$\\
\hphantom{$-$}3.0  &  $(0.1486,0.7817,0.06969)$             &  $(0.4542,2.387,0.2113)\hphantom{0}$ &               $-0.3163$\hphantom{0} & $-0.6142$\\
\hphantom{$-$}3.5  &  $(0.1771,0.7672,0.05573)$             &  $(0.5275,2.281,0.1644)\hphantom{0}$ &               $-0.4867$\hphantom{0} & $-0.6274$\\
\hphantom{$-$}4.0  &  $(0.2090,0.7470,0.04405)$             &  $(0.6057,2.159,0.1261)\hphantom{0}$ &               $-0.6618$\hphantom{0} & $-0.6407$\\
\hphantom{$-$}4.5  &  $(0.2442,0.7214,0.03440)$             &  $(0.6876,2.024,0.09540)$            &               $-0.8425$\hphantom{0} & $-0.6539$\\
\noalign{\smallskip}
\hline
\noalign{\smallskip}
\end{tabular}}
\end{footnotesize}
\end{center}
\end{table}

\section{Conclusions}

A tilted flashing Brownian ratchet is a flashing Brownian ratchet in the presence of a static homogeneous force acting in the direction opposite that of the directed motion, thereby reducing (or reversing) the directed motion effect.  We began by deriving a general class of capital-dependent Parrondo games with a bias parameter, motivated by the tilted Brownian ratchet.  These Parrondo games, in turn, motivated our random walk approximation of the tilted flashing Brownian ratchet, which we subsequently improved.  This is an efficient method of numerically approximating the distribution of a continuous process at a fixed time, more accurate than stochastic simulation (Section~\ref{EM}) but equivalent to numerical PDE analysis (Section~\ref{FP}).  We found the approximate density of the tilted flashing Brownian ratchet after one time period, starting at 0. We also found the approximate density of the tilted flashing Brownian ratchet after the same time period, but now starting from a stationary distribution associated with the wrapped tilted flashing Brownian ratchet.  In both cases the distributions are mixtures of asymmetric unimodal distributions, contrary to what was suggested by conceptual figures such as Figure~\ref{HAT00fig}.  We approximated the mean displacement of the tilted flashing Brownian ratchet over that time period and observed that the mean displacement is increasing in $\gamma$, a measure of the amplitude of the potential, and decreasing in $\kappa$, the slope of tilt.

\section*{Acknowledgments}

The work of S. N. Ethier was partially supported by a grant from the Simons Foundation (429675). 
The work of J. Lee was supported by the Basic Science Research Program through the National Research Foundation of Korea (NRF) funded by the Ministry of Education (NRF-2018R1D1A1B07042307).

\end{document}